\documentclass[11pt]{amsart}

\usepackage{amsmath}
\usepackage{amssymb}
\usepackage{graphicx}

\newtheorem{theorem}{Theorem}[section]
\newtheorem{corollary}[theorem]{Corollary}
\newtheorem{lemma}[theorem]{Lemma}
\newtheorem{proposition}[theorem]{Proposition}
\theoremstyle{definition}
\newtheorem{definition}[theorem]{Definition}
\newtheorem{example}[theorem]{Example}
\newtheorem{remark}[theorem]{Remark}

\numberwithin{equation}{section}

\title[A Dual Representation Theorem on the Conditional Orlicz Space ]{a dual representation theorem on the conditional Orlicz space generated from a random normed module}

\author[X. Zhang]{Xia Zhang}

\address[X. Zhang]{School of Mathematical Sciences, TianGong University, Tianjin 300387, P.R.China}
\email{{\tt zhangxia@tiangong.edu.cn}}

\author[K. Qian]{Ke Qian}
\address[K. Qian]{School of Mathematical Sciences, TianGong University, Tianjin 300387, P.R.China}
\email{\tt qianke2230111367@163.com}

\author[M. Liu]{Ming Liu$^{*}$}
\address[M. Liu]{School of Mathematical Sciences, TianGong University, Tianjin 300387, P.R.China}
\email{\tt liuming@tiangong.edu.cn}

\thanks{$^*$Corresponding author.}



\begin{document}

\begin{abstract}
In this paper, we first introduce the notion of a random Orlicz function, and further present  the conditional Orlicz space generated from a random normed module. Second, we prove the denseness of the Orlicz heart of a random normed module $E$ in $E$ with respect to the $(\varepsilon, \lambda)$-topology. Finally, based on the above work, we establish a dual representation theorem on the conditional Orlicz space generated from a random normed module, which extends and improves some known results.
\end{abstract}

\maketitle


\section{Introduction}
A random normed module (briefly, an $RN$ module) is a random generalization of a classical normed space. The original definition of $RN$ modules was first introduced in [14], with a more detailed and refined version presented in [13]. The notion of $RN$ modules was initially introduced to address the limitations of normed spaces in dealing with randomness and uncertainty. The main difficulty in the study of $RN$ modules is that the development and complexity of $RN$ modules make the traditional theory of conjugate spaces less applicable. This needs new approaches and theories to handle the above difficulty. The significant breakthrough in the field of $RN$ modules was made by Guo, attributed to his establishment of the Hahn-Banach theorem for an almost everywhere bounded random linear functional [12]. Furthermore, the representation theory of random conjugate spaces and characterizations of random reflexivity within the framework of random conjugate spaces have yielded many profound results. For example, Guo [15] advanced the development of this field by proving the representation theorems for the dual of Lebesgue-Bochner function spaces. Lebesgue-Bochner function spaces generalize Lebesgue spaces by allowing functions to take values in a Banach space rather than just real or complex numbers. The representation theorems of the dual space are essential for understanding the structure of these function spaces. Subsequently, Guo and Li [16] established the James theorem in complete $RN$ modules, extending a fundamental result in Banach space theory. The James theorem characterizes reflexive Banach spaces, and Guo and Li's work generalizes this theorem to $RN$ modules, which is significant for dealing with complex systems involving randomness.
Recently, we were delight to find out about some excellent work with respect to the dual representation theorem of Gao et al. For example, Gao and Xanthos [8] deeply studied the C-property and $w^*$-representations of risk measures, and Gao, Leung and Xanthos [7] provided that a proper convex functional satisfying the Fatou property on a general function space must not admit a tractable dual representation for the class of Orlicz spaces. Wu, Long and Zeng [3] established a representation theorem which identify the dual of the Orlicz heart of an $RN$ module with the Orlicz space generated from the random conjugate space.
Let $E$ be an $RN$ module with base $(\Omega, \mathcal{E}, P)$, $\mathcal{F}$ a sub-$\sigma$-algebra of $\mathcal{E}$, and $1 \leq p< +\infty$ and $1 < q\leq +\infty$ a pair of the $H\ddot{o}lder$ conjugate numbers, then the space $L_{\mathcal{F}}^{p}(E)$ generated from $E$ is defined by
$$L_{\mathcal{F}}^{p}(E)=\left\{x \in E:|||x|||_{p} \in L_{+}^{0}(\mathcal{F})\right\},$$
which can be made into an $RN$ module in a natural way. Guo $[12]$ proved a basic representation theorem which says that the random conjugate space of $L_{\mathcal{F}}^{p}(E)$ is isometrically isomorphic to $L_{\mathcal{F}}^q\left(E^*\right)$. Wu and Zhao [4] gave a proof of Stricker's lemma based on a result in the theory of $RN$ modules.

The theory of risk measures has always been an active and fruitful research area in the finance. Gao, Leung, Munari and Xanthos [6] studied all kinds of results for quasiconvex, law-invariant functionals on a general Orlicz space. The above results have wide applications in the theory of risk measures. Chen, Gao, Leung and Li [10] investigated automatic Fatou property of law-invariant risk measures on a rearrangement-invariant function space. Moreover, the theory of $RN$ modules has also been applied to conditional risk measures [11,17]. The $RN$ module $L_{\mathcal{F}}^{p}(\mathcal{E})$ could serve as the most general model space for the unified study of various types of conditional risk measures and the scholars established an extension theorem for $L^p$-conditional risk measures. We refer to $[1, 2, 6, 9]$ about Orlicz spaces and risk measures. Motivated by Gao and Wu et al., this paper is focused on introducing and studying the conditional Orlicz space generated from an $RN$ module.

Throughout this paper, we first introduce the notion of a random Orlicz function, and further present the conditional Orlicz space generated from an $RN$ module. Second, we prove the denseness of the Orlicz heart of an $RN$ module $E$ in $E$ with respect to the $(\varepsilon, \lambda)$-topology, which is important in the proof of main theorems. Finally, we establish a dual representation theorem on the conditional Orlicz space generated from an $RN$ module, which says that the dual of the conditional Orlicz heart generated from an $RN$ module is the conditional Orlicz space generated from the random conjugate space. These results extend and improve some known results.

\section{Terminology and Notation}
Let $(\Omega, \mathcal{F}, P)$ be a probability space, $K$ the scalar field of real numbers $\mathbb{R}$ or complex numbers $\mathbb{C}$, $\bar{L}^0(\mathcal{F}, \mathbb{R})$ the set of equivalence classes of $\mathcal{F}$-measurable extended real-valued random variables on $\Omega$, and $L^0(\mathcal{F}, K)$ the algebra of equivalence classes of $\mathcal{F}$-measurable $K$-valued random variables on $\Omega$. $L^0(\mathcal{F}, \mathbb{R})$ and $\bar{L}^0(\mathcal{F}, \mathbb{R})$ are written as $L^0(\mathcal{F})$ and $\bar{L}^0(\mathcal{F})$, respectively.
Specially,
$$L_{+}^0(\mathcal{F})=\left\{\zeta \in L^0(\mathcal{F}) \mid \zeta \geq 0\right\}, L_{++}^0(\mathcal{F})=\left\{\zeta \in L^0(\mathcal{F}) \mid \zeta > 0\right\}.$$

Besides, $\bar{L}^0(\mathcal{F})$ is partially ordered by $\xi \leq \eta$ iff $\xi^0(\omega) \leq \eta^0(\omega)$ for $P$-almost all $\omega \in \Omega$, where $\xi^0$ and $\eta^0$ are arbitrarily chosen representatives of $\xi$ and $\eta$, respectively.

According to [5], $\left(\bar{L}^0(\mathcal{F}), \leq\right)$ is a complete lattice and $\left(L^0(\mathcal{F}), \leq\right)$ is a conditionally complete lattice. For every subset $H$ in $\bar{L}^0(\mathcal{F}), \vee H$ represents the supremum of $H$. If $H$ is upward directed, namely there exists $z \in H$ for any $x, y \in H$ such that $x \leq z$ and $y \leq z$, then there exists a sequence $\left\{z_n: n \in \mathbb{N}\right\}$ in $H$ such that $\left\{z_n: n \in \mathbb{N}\right\}$ converges to $\vee H$ in a nondecreasing way. In addition, if $I_A$ denotes the characteristic function of $A$ for any $A \in \mathcal{F}$, then $\tilde{I}_A$ is the equivalence class of $I_A$.

For any $\zeta \in L^0(\mathcal{F}),|\zeta|$ denotes the equivalence class of $\left|\zeta^0\right|: \Omega \rightarrow[0, \infty)$ defined by $$\left|\zeta^0\right|(\omega)=\left|\zeta^0(\omega)\right|$$ for an arbitrarily chosen representative $\zeta^0$ of $\zeta$.

For any $x \in L^0(\mathcal{F}, K)$, ${sgn}(x)$ denotes the equivalence class of ${sgn}\left(x^0\right)$ defined by
$$
{sgn}\left(x^0\right)(\omega)=\left\{\begin{array}{ll}
\frac{\left|x^0(\omega)\right|}{x^0(\omega)}, & x^0(\omega) \neq 0, \\
 0, & x^0(\omega) = 0,
\end{array}\right.
$$
where $x^0$ is an arbitrarily chosen representative of $x$.

\begin{definition}[\cite{13}]
An ordered pair $(E,\|\cdot\|)$ is called a random normed module (briefly, an $RN$ module) over $K$ with base $(\Omega, \mathcal{F}, P)$ if $E$ is a left module over $L^0(\mathcal{F}, K)$, and the mapping $\|\cdot\|: E \rightarrow L_{+}^0(\mathcal{F})$ satisfies:

($1$)\, $\|x\|=0$ iff $x=\theta$ ( the null element of $E$ );

($2$)\, $\|\zeta x\|=|\zeta|\|x\|$ for all $\zeta \in L^0(\mathcal{F}, K)$ and $x \in E$;

($3$)\, $\left\|x+y\right\| \leq\left\|x\right\|+\left\|y\right\|$ for all $x, y \in E$.
\end{definition}

Suppose that $(E,\|\cdot\|)$ is an $RN$ module, $E$ is always endowed with the $(\varepsilon, \lambda)$-topology. Then a sequence $\left\{x_n: n \in \mathbb{N}\right\}$ in $E$ converges in the $(\varepsilon, \lambda)$-topology to $x\in E$ iff the sequence $\left\{\left\|x_n-x\right\|: n \in \mathbb{N}\right\}$ in $L_{+}^0(\mathcal{F})$ converges in probability to 0. Specially, $\left(L^0(\mathcal{F}, K),|\cdot|\right)$ is an $RN$ module.

In addition, $E^*$ denotes the $L^0(\mathcal{F}, K)$-module of all continuous module homomorphisms $f$ from $(E, \|\cdot\|)$ to $\left(L^0(\mathcal{F}, K), |\cdot|\right)$. It follows from [14] that, if $f: E \rightarrow L^0(\mathcal{F}, K)$ is a linear mapping, then $f \in E^*$ iff $f$ is almost surely bounded, which means that for some $\zeta \in L_{+}^0(\mathcal{F}), |f(x)| \leq \zeta\|x\|$ holds for all $x \in E$. Define $$\|f\|^*=\vee\{|f(x)|: x \in E, \|x\| \leq 1\}$$ for any $f \in E^*$, then $\left(E^*, \|\cdot\|^*\right)$ is an $RN$ module, called the random conjugate space of $(E, \|\cdot\|)$.

\begin{definition}[\cite{1}]
A function $\phi:[0, \infty) \rightarrow[0, \infty]$ is called an Orlicz function if it satisfies:

$(1)~ \phi(0)=0$;

$(2)~ \phi$ is left-continuous;

$(3)~ \phi$ is increasing;

$(4)~ \phi$ is convex;

$(5)~ \phi$ is nontrivial.
\end{definition}

We observe that $\phi$ is continuous except possibly at a single point, where it jumps to $\infty$. Define the conjugate of $\phi$ by
$$\psi(s):=\sup _{t \geq 0}\{t s-\phi(t)\}$$
for any $s \geq 0$, then $\psi$ is also an Orlicz function. It is easy to check that the conjugate of $\psi$ is $\phi$.

In Definition $2.3$ below, we first introduce the notion of a random Orlicz function, which is a natural generalization of an Orlicz function.
\begin{definition}
A function $\Phi:L_{+}^{0}(\mathcal{F}) \rightarrow \bar{L}_{+}^{0}(\mathcal{F})$ is called a random Orlicz function if it satisfies:

$(1)~ \Phi(0)=0$;

$(2)~ \Phi$ is left-continuous: $\forall \varepsilon>0, \lim\limits_{t\uparrow t_0} P\{\omega \in \Omega:|\Phi(t)-\Phi(t_0)|(\omega)>\varepsilon\}=0$;

$(3)~ \Phi$ is increasing: $\Phi(t_1) \geq \Phi(t_2)$ for all $t_1, t_2 \in L_{+}^{0}(\mathcal{F})$ such that $t_1 \geq t_2$;

$(4)~ \Phi$ is $L^{0}(\mathcal{F})$-convex: $\Phi(\zeta t_1+(1-\zeta)t_2)\leq \zeta\Phi(t_1)+(1-\zeta)\Phi(t_2)$ for any $\zeta \in L_{+}^0(\mathcal{F})$ with $0 \leqslant \zeta \leqslant 1$ and $x_1, x_2 \in L_{+}^{0}(\mathcal{F})$;

$(5)~ \Phi(t)\in L_{++}^{0}(\mathcal{F})$ for some $t \in L_{++}^{0}(\mathcal{F})$.
\end{definition}

Define the random conjugate of $\Phi$ by
$$\Psi(s):=\vee \{t s-\Phi(t):t \in L_{+}^{0}(\mathcal{F})\}$$
for any $s \in L_{+}^{0}(\mathcal{F})$, then $\Psi$ is also a random Orlicz function. It is easy to check that the random conjugate of $\Psi$ is $\Phi$.

In addition, $\Phi$ is said to satisfy the $\triangle_2$-condition, defined by $\Phi \in \triangle_2$, if there is $\zeta \in L_{++}^{0}(\mathcal{F})$ such that $\Phi(2t) \leq \zeta \Phi(t)$ for every $t \in L_{+}^{0}(\mathcal{F})$.
\section{The Conditional Orlicz Space Generated from a Random Normed Module}
In this paper, $(\Omega, \mathcal{E}, P)$ denotes a probability space, $\mathcal{F}$ a sub-$\sigma$-algebra of $\mathcal{E}$, $\phi$ an Orlicz function and $\Phi$ a random Orlicz function.
The Orlicz space with respect to $\phi$ is defined by
$$
L^{\phi}=\left\{\zeta \in L^0(\mathcal{F}): E[\phi(d|\zeta|)]<\infty \text { for some } d>0\right\}
$$
and the Orlicz heart is defined by
$$
H^{\phi}=\left\{\zeta \in L^0(\mathcal{F}): E[\phi(d|\zeta|)]<\infty \text { for all } d>0\right\},
$$
where $E[\zeta]$ denotes $\zeta$'s expectation as to the probability $P$.

Denote
$$
|\zeta|_{\phi L}=inf \left\{\lambda>0 : E\left[\phi\left(\frac{|\zeta|}{\lambda}\right)\right] \leq 1\right\}
$$
and
$$
|\zeta|_{\phi O}=sup \left\{|E[\zeta \eta]|: \eta \in L^{\psi}, |\eta|_{\psi L} \leq 1\right\}.
$$
The conditional Orlicz space with respect to $\Phi$ is defined by
$$
L_{\mathcal{F}}^{\Phi}=\left\{\zeta \in L^0(\mathcal{F}): E[\Phi(d|\zeta|)|\mathcal{F}]\in L^{0} ({\mathcal{F}})\text { for some } d \in L_{++}^{0} ({\mathcal{F}})\right\}
$$
and the conditional Orlicz heart is defined by
$$
H_{\mathcal{F}}^{\Phi}=\left\{\zeta \in L^0(\mathcal{F}): E[\Phi(d|\zeta|)|\mathcal{F}]\in L^{0} ({\mathcal{F}})\text { for all } d\in L_{++}^{0} ({\mathcal{F}})\right\},
$$
where $E[\zeta|\mathcal{F}]$ denotes $\zeta$'s conditional expectation as to $\mathcal{F}$.

Denote
$$
|\zeta|_{\Phi L}=\wedge \left\{\lambda\in L_{++}^{0} ({\mathcal{F}}): E\left[\Phi\left(\frac{|\zeta|}{\lambda}\right)\big|\mathcal{F}\right] \leq 1\right\}
$$
and
$$
|\zeta|_{\Phi O}=\vee \left\{|E[\zeta \eta|\mathcal{F}]|: \eta \in L_{\mathcal{F}}^{\Psi}, |\eta|_{\Psi L} \leq 1\right\}.
$$

\begin{definition}\label{def-COSE}
Let $(E,\|\cdot\|)$ be an $RN$ module and $\Phi$ a random Orlicz function. Then
$$
L_{\mathcal{F}}^{\Phi}(E):=\left\{x \in E:\|x\| \in L_{\mathcal{F}}^{\Phi}\right\}
$$
is called the conditional Orlicz space generated from $E$ with respect to $\Phi$ and
$$
H_{\mathcal{F}}^{\Phi}(E):=\left\{x \in E:\|x\| \in H_{\mathcal{F}}^{\Phi}\right\}
$$
is called the conditional Orlicz heart of $E$.
\end{definition}

Furthermore, if we define $$|||x|||_{\Phi L}=\big|\|x\|\big|_{\Phi L}$$ and
$$|||x|||_{\Phi O}=\big|\|x\|\big|_{\Phi O}$$ for all $x \in L_{\mathcal{F}}^{\Phi}(E)$. Then one can easily show that $(L_{\mathcal{F}}^{\Phi}(E),|||\cdot|||_{\Phi L})$ and $(H_{\mathcal{F}}^{\Phi}(E),|||\cdot|||_{\Phi L})$ are $RN$ modules.

In the following, we use $|||\cdot|||_{p}$ to denote the norm on $L_{\mathcal{F}}^p(E)$ for all $p\in [1,+\infty]$.
\begin{example}
Let $(E,\|\cdot\|)$ be an $RN$ module with base $(\Omega, \mathcal{E}, P)$, $\mathcal{F}$ a sub-$\sigma$-algebra of $\mathcal{E}$, and $\Phi$ a random Orlicz function.

1. If $\Phi(t)=t$, then
$$
\Psi(s)=\left\{\begin{array}{ll}
0, & 0\leq s\leq 1, \\
 \infty, & s>1.
\end{array}\right.
$$
Further, we have
$$
\begin{array}{l}
H_{\mathcal{F}}^{\Phi}(E)=L_{\mathcal{F}}^{\Phi}(E)=L_{\mathcal{F}}^{1}(E), \,|||\cdot|||_{\Phi L}=|||\cdot|||_{\Phi O}=|||\cdot|||_{1}, \\
L_{\mathcal{F}}^{\Psi}(E)=L_{\mathcal{F}}^{\infty}(E), \,H_{\mathcal{F}}^{\Psi}(E)=\{0\}, \,|||\cdot|||_{\Psi O}=|||\cdot|||_{\Psi L}=|||\cdot|||_{\infty} .
\end{array}
$$

2. Let $p, q \in(1, \infty)$ be a pair of the $H\ddot{o}lder$ conjugate numbers. If $\Phi(t)=t^p$, then $\Psi(s)=p^{1-q} q^{-1} s^q$. Further, we have
$$
\begin{array}{l}
H_{\mathcal{F}}^{\Phi}(E)=L_{\mathcal{F}}^{\Phi}(E)=L_{\mathcal{F}}^{p}(E), \,|||\cdot|||_{\Phi L}=|||\cdot|||_{p}, \,|||\cdot|||_{\Phi O}=p^{\frac{1}{p}} q^{\frac{1}{q}}|||\cdot|||_{p}, \\
H_{\mathcal{F}}^{\Psi}(E)=L_{\mathcal{F}}^{\Psi}(E)=L_{\mathcal{F}}^{q}(E), \,|||\cdot|||_{\Psi O}=|||\cdot|||_{q}, \,|||\cdot|||_{\Psi L}=p^{-\frac{1}{p}} q^{-\frac{1}{q}}|||\cdot|||_{q}.
\end{array}
$$
\end{example}

In 2022, Wu, Long and Zeng introduced the Orlicz space generated from an $RN$ module $E$ with respect to an Orlicz function $\phi$ defined by
$$
L^{\phi}(E)=\left\{x \in E:\|x\| \in L^{\phi}\right\}
$$
and the Orlicz heart of $E$ defined by
$$
H^{\phi}(E)=\left\{x \in E:\|x\| \in H^{\phi}\right\}.
$$
Define $$||x||_{\phi L}=\big|\|x\|\big|_{\phi L}$$ and
$$||x||_{\phi O}=\big|\|x\|\big|_{\phi O}$$ for all $x \in L^{\phi}(E)$. Wu shows that $(L^{\phi}\left(E\right),||\cdot||_{\phi L})$ and $(H^{\phi}(E),||\cdot||_{\phi L})$ are normed spaces, and if $E$ is complete then
$(L^{\phi}\left(E\right),||\cdot||_{\phi L})$ and $(H^{\phi}(E),||\cdot||_{\phi L})$ are also complete. Now we will show that its converse is also true.

\begin{proposition}
Let $(E,\|\cdot\|)$ be an $RN$ module, $\phi$ a continuous Orlicz function, and $(L^{\phi}\left(E\right),||\cdot||_{\phi L})$ and $(H^{\phi}(E),||\cdot||_{\phi L})$ defined as above. If $(L^{\phi}\left(E\right),||\cdot||_{\phi L})$ and $(H^{\phi}(E),||\cdot||_{\phi L})$ are complete, then $E$ is also complete.
\end{proposition}
\begin{proof}
Let $\left\{x_n:n \in \mathbb{N}\right\}$ be a Cauchy sequence in $E$. Then $\left\{\left\|x_n\right\|: n \in \mathbb{N}\right\}$ is a Cauchy sequence in $L^0(\mathcal{E})$. Hence there exists a subsequence $\left\{\left\|x_{n_k}\right\|: k \in \mathbb{N}\right\}$ of $\left\{\left\|x_n\right\|: n \in \mathbb{N}\right\}$ such that $\left\{\left\|x_{n_k}\right\|: k \in \mathbb{N}\right\}$ converges $P-a.s.$ to some $\xi \in L^0(\mathcal{E})$.

Set $\zeta=\vee_{k \geq 1}\left\|x_{n_k}\right\|+1 \in L_{++}^0(\mathcal{E})$ and $z_k=\zeta^{-1} x_{n_k}$ for any $k \in \mathbb{N}$, then $\left\{z_k: k \in \mathbb{N}\right\}$ is a Cauchy sequence in $E$. It is obvious that $\left\|z_k\right\|<1$, then $z_k \in L^{\phi}(E)$.

According to Lebesgue's dominance convergence theorem, we have that $\left\{z_k: k \in \mathbb{N}\right\}$ is a Cauchy sequence in $L^{\phi}(E)$. Then using the fact that $L^{\phi}(E)$ is complete, $\left\{z_k: k \in \mathbb{N}\right\}$ converges to some $z \in L^{\phi}(E)$, which implies that $\left\{z_k: k \in \mathbb{N}\right\}$ converges to some $z \in E$. Thus we can easily see that $\left\{x_n: n \in \mathbb{N}\right\}$ converges to some $\zeta z$.

Similarly, if $H^{\phi}(E)$ is complete, then we can also get that $E$ is complete.
\end{proof}

Based on Proposition 3.3, we have the following conclusion.

\begin{proposition}
Let $(E,\|\cdot\|)$ be an $RN$ module, $\phi$ a continuous Orlicz function, $\Phi$ a random Orlicz function and $(L_{\mathcal{F}}^{\Phi}(E),|||\cdot|||_{\Phi L})$ and $(H_{\mathcal{F}}^{\Phi}(E),|||\cdot|||_{\Phi L})$ defined as above. Then $(L_{\mathcal{F}}^{\Phi}(E),|||\cdot|||_{\Phi L})$ and $(H_{\mathcal{F}}^{\Phi}(E),|||\cdot|||_{\Phi L})$ are complete if and only if $E$ is complete.
\end{proposition}
\begin{proof}
Note that
$$
\begin{aligned}
L^{\phi}(L_{\mathcal{F}}^{\Phi}(E))& =\left\{x \in L_{\mathcal{F}}^{\Phi}(E):|||x|||_{\Phi L} \in L^{\phi}\right\}\\
&=\left\{x \in E:\big|\|x\|\big|_{\Phi L} \in L^{\phi}\right\}\\
&=\left\{x \in E:\|x\|\in L^{\phi}\right\},
\end{aligned}
$$
thus we have $L^{\phi}(L_{\mathcal{F}}^{\Phi}(E))=L^{\phi}(E).$ According to Proposition 3.3 and [3, Proposition 3.3], we have that
$$L_{\mathcal{F}}^{\Phi}(E)\,\text{is complete iff}\, L^{\phi}\left(L_{\mathcal{F}}^{\Phi}(E)\right)=L^{\phi}(E)\,\text{is complete iff}\, E \,\text{is complete}.$$

Similarly, we can prove that $H_{\mathcal{F}}^{\Phi}\left(E\right)$ is a complete $RN$ module.
\end{proof}

Moreover, we obtain the denseness of the Orlicz heart of $E$ in $E$ with respect to the $(\varepsilon, \lambda)$-topology, which is important in the proofs of main theorems in sequel.

\begin{proposition}\label{den-OHE}
Let $(E,\|\cdot\|)$ be an $RN$ module and $\phi$ a continuous Orlicz function. Then $H^{\phi}(E)$ is dense in $E$ with respect to the $(\varepsilon, \lambda)$-topology.
\end{proposition}
\begin{proof}
Let $x \in E$ and $A_n=\{\omega\in\Omega:  \|x\|(\omega)\leq n\}$. Set $x_n=\tilde{I}_{A_n} \cdot x$ for each $n \in \mathbb{N}$. Since $x_n=\tilde{I}_{A_n} \cdot x+\left(1-\tilde{I}_{A_n}\right) \cdot \theta$ and $\theta \in E$, we have $x_n \in E$. Then $x_n \in H^{\phi}(E)$ follows $\|x_{n}\| \in H^{\phi}$. Since $$\left\|x-x_n\right\|=\left(1-\tilde{I}_{A_n}\right)\|x\| \rightarrow 0$$ as $n \rightarrow \infty$, it follows that $\left\{x_n: n \in \mathbb{N}\right\}$ converges to $x$ in the $(\varepsilon, \lambda)$-topology.
\end{proof}

In 2022, Wu, Long and Zeng presented the following dual representation theorem.
\begin{proposition}[\cite{3}]\label{dual-OE}
Let $(E,\|\cdot\|)$ be an $RN$ module and $\phi$ a continuous Orlicz function with conjugate $\psi$. $\left(E^*,\|\cdot\|^*\right)$ denotes the random conjugate space of $\left(E,\|\cdot\|\right)$. Then
$$
\left(H^{\phi}(E),\|\cdot\|_{\phi L}\right)^{\prime} \cong\left(L^{\psi}\left(E^*\right),\|\cdot\|_{\psi O}\right),
$$
where the isometric isomorphism $T:\left(L^{\psi}\left(E^*\right),\|\cdot\|_{\psi O}\right) \rightarrow\left(H^{\phi}(E),\|\cdot\|_{\phi L}\right)^{\prime}$ is given by
$$
[T f](x)=E[f(x)]
$$
for each $f \in L^{\psi}\left(E^*\right)$ and $x \in H^{\phi}(E)$.
\end{proposition}

\section{Dual Representation of the Random Conjugate Space of $H_{\mathcal{F}}^{\Phi}(E)$}
Our main results are as follows.
\begin{theorem}
Let $(E,\|\cdot\|)$ be an $RN$ module with base $(\Omega, \mathcal{E}, P)$, $\mathcal{F}$ a sub-$\sigma$-algebra of $\mathcal{E}$, and $\Phi$ a random Orlicz function with the random conjugate $\Psi$. $\left(E^*,\|\cdot\|^*\right)$ denotes the random conjugate space of $\left(E,\|\cdot\|\right)$. Then
$$
\left(H_{\mathcal{F}}^{\Phi}(E), |||\cdot|||_{\Phi L}\right)^{\ast} \cong\left(L_{\mathcal{F}}^{\Psi}\left(E^*\right), |||\cdot|||_{\Psi O}\right),
$$
where the isometric isomorphism $T:\left(L_{\mathcal{F}}^{\Psi}\left(E^*\right), |||\cdot|||_{\Psi O}\right) \rightarrow\left(H_{\mathcal{F}}^{\Phi}(E), |||\cdot|||_{\Phi L}\right)^{\ast}$ is given by
$$
[T f](x)=E[f(x)|\mathcal{F}]
$$
for each $f \in L_{\mathcal{F}}^{\Psi}\left(E^*\right)$ and $x \in H_{\mathcal{F}}^{\Phi}(E)$.
\end{theorem}

We will divide the proof of Theorem 4.1 into the following two Lemmas 4.2 and 4.3.

\begin{lemma}
$T$ is well defined and isometric.
\end{lemma}
\begin{proof}
For any fixed $f \in L_{\mathcal{F}}^{\Psi}\left(E^*\right)$, we will prove $T f \in\left(H_{\mathcal{F}}^{\Phi}(E),|||\cdot|||_{\Phi L}\right)^{\ast}$ and $\|T f\|=|||f|||_{\Psi O}$.
For any $x \in H_{\mathcal{F}}^{\Phi}(E)$, we have:
$$
\begin{aligned}
|[T f](x)|
&=|E[f(x)|\mathcal{F}]| \\
&\leq E\left[\|f\|^*\|x\|\right|\mathcal{F}] \\
&\leq\big|\|f\|^*\big|_{\Psi O} \big|\|x\|\big|_{\Phi L}\\
&=|||f|||_{\Psi O}|||x|||_{\Phi L},
\end{aligned}
$$
which shows that $T f \in\left(H_{\mathcal{F}}^{\Phi}(E),|||\cdot|||_{\Phi L}\right)^{\ast}$ and $\|T f\|\leq |||f|||_{\Psi O}$, namely $T$ is well defined.

Next, we remain to show $\|T f\|\geq|||f|||_{\Psi O}$.
Observing that
$$
\begin{aligned}
|||f|||_{\Psi O}
&=\big|\|f\|^*\big|_{\Psi O}\\
&=\vee \left\{|E\left[\|f\|^* \zeta|\mathcal{F}\right]|:\zeta \in H_{\mathcal{F}}^{\Phi}, |\zeta|_{\Phi L} \leq 1\right\}
\\
&=\vee \left\{E\left[\|f\|^* \zeta|\mathcal{F}\right]: \zeta \in H_{\mathcal{F}}^{\Phi}, \zeta\geq 0, |\zeta|_{\Phi L} \leq 1\right\},
\end{aligned}
$$
we only need to prove $\|T f\| \geq E\left[\|f\|^* \zeta|\mathcal{F}\right]$ for any fixed $\zeta \in H_{\mathcal{F}}^{\Phi}$ with $\zeta\geq 0$ and $|\zeta|_{\Phi L} \leq 1$.

Moreover, we can easily verify that the set $\{|f(x)|: x \in U(E)\}$ is upward directed,
where the random closed unit ball of $E$ is defined by $$U(E):=\{x\in E:\|x\| \leq 1\}.$$ Thus there exists a sequence $\left\{x_n: n \in \mathbb{N}\right\}$ in $U(E)$ such that $\left\{\left|f\left(x_n\right)\right|: n \in \mathbb{N}\right\}$ converges to $\vee\{|f(x)|: x \in U(E)\}=\|f\|^*$ in a nondecreasing way.

Suppose that $f\left(x_n\right)=\left|f\left(x_n\right)\right|$ for each $n$, otherwise each $x_n$ will be replaced with $\left({sgn f}\left(x_n\right)\right) x_n$. Hence we have
$$\lim _{n \rightarrow \infty} f\left(\zeta x_n\right)=\lim _{n \rightarrow \infty} \zeta f\left(x_n\right)=\zeta\|f\|^*.$$

Finally, since $\left\|\zeta x_n\right\|=\zeta\left\|x_n\right\| \leq \zeta$ for all $x_n \in U(E)$ and $\zeta\geq 0$, it follows that $$|||\zeta x_n|||_{\Phi L}=\big|\|\zeta x_n\|\big|_{\Phi L} \leq|\zeta|_{\Phi L} \leq 1.$$
Thus $$E\left[f\left(\zeta x_n\right)|\mathcal{F}\right]=[T f]\left(\zeta x_n\right) \leq\|T f\|.$$
Consequently, we obtain that $$E\left[\|f\|^* \zeta|\mathcal{F}\right]=\lim _{n \rightarrow \infty} E\left[f\left(\zeta x_n\right)|\mathcal{F}\right] \leq\|T f\|.$$
\end{proof}

\begin{lemma}
$T$ is surjective.
\end{lemma}

\begin{proof}
Suppose that $F \in \left(H_{\mathcal{F}}^{\Phi}(E),|||\cdot|||_{\Phi L}\right)^{\ast}$. We will prove that there exists an $f \in$ $L_{\mathcal{F}}^{\Psi}\left(E^*\right)$ such that $F=Tf$.

Let $\zeta$ be a given representative of $\|F\|$ and $A_n=\{\omega \in \Omega :$ $n-1 \leq \zeta(\omega)<n\}$ for each $n \in \mathbb{N}$. Since $|F(x)| \leq\|F\|\cdot|||x|||_{\Phi L}$, it follows that $$\left|\tilde{I}_{A_n} F(x)\right| \leq n|||x|||_{\Phi L}$$ for all $x \in H_{\mathcal{F}}^{\Phi}(E)$ and $n \in \mathbb{N}$.

Next, let $n$ be fixed, note that
$$\left|\int_{\Omega}\tilde{I}_{A_n} F(x) d P\right| \leq \int_{\Omega}\left|\tilde{I}_{A_n} F(x)\right| d P \leq n \int_{\Omega}|||x|||_{\Phi L} d P$$
for all $x \in H^{\phi}(E)$. According to Proposition \ref{dual-OE}, we have that there exists an $f_n \in L^{\psi}\left(E^*\right)$ such that $$\int_{\Omega}\tilde{I}_{A_n} F(x) d P=\int_{\Omega} f_n(x) d P$$ for all $x \in H^{\phi}(E)$.

Since $\tilde{I}_A H^{\phi}(E) \subset H^{\phi}(E)$ for all $A \in \mathcal{F}$, it follows that
$$\int_A\tilde{I}_{A_n} F(x) d P=\int_A f_n(x) d P=\int_A E\left[f_n(x) \mid \mathcal{F}\right] d P$$
for all $x \in H^{\phi}(E)$ and $A \in \mathcal{F}$. Then by observing $\tilde{I}_{A_n} F(x) \in L^0(\mathcal{F}, K)$, we have that
$\tilde{I}_{A_n} F(x)=E\left[f_n(x) \mid \mathcal{F}\right]$ for all $x \in H^{\phi}(E)$.

Since $f_n \in L^{\psi}\left(E^*\right) \subset L_{\mathcal{F}}^{\Psi}\left(E^*\right)$, we have that $T{f_n} \in\left(H_{\mathcal{F}}^{\Phi}(E),|||\cdot|||_{\Phi L}\right)^{\ast}$, which shows that $\tilde{I}_{A_n} F$ and $T{f_n}$ are equal on $H^{\phi}(E)$.

Subsequently, according to Proposition \ref{den-OHE}, $H^{\phi}(H_{\mathcal{F}}^{\Phi}(E))=H^{\phi}(E)$ is dense in $H_{\mathcal{F}}^{\Phi}(E)$ with respect to the $(\varepsilon, \lambda)$-topology, namely we have that $\tilde{I}_{A_n} F(x)=E\left[f_n(x) \mid \mathcal{F}\right]$ for all $x \in H_{\mathcal{F}}^{\Phi}(E)$.

Furthermore, we can get that $\tilde{I}_{A_n} F(x)=E\left[\tilde{I}_{A_n} f_n(x) \mid \mathcal{F}\right]$ for all $x \in H_{\mathcal{F}}^{\Phi}(E)$.

Finally, set $f(x)=\sum\limits_{n \geq 1} \tilde{I}_{A_n} f_n(x)$ for all $x \in H_{\mathcal{F}}^{\Phi}(E)$, since $L_{\mathcal{F}}^{\Psi}\left(E^*\right)$ has the countable concatenation property, it follows that $f \in$ $L_{\mathcal{F}}^{\Psi}\left(E^*\right)$. Then we have that $$F(x)=E[f(x) \mid \mathcal{F}]=[Tf](x)$$ for all $x \in H_{\mathcal{F}}^{\Phi}(E)$.
\end{proof}

\begin{remark}
If we take $\Phi(t)=t^p$ for all $p \in[1, \infty)$ in Theorem 4.1, then we obtain
$$
\left(L_{\mathcal{F}}^p(E),|||\cdot|||_p\right)^{*} \cong\left(L_{\mathcal{F}}^q\left(E^*\right),|||\cdot|||_q\right) .
$$
Thus the results in this section extend one of the results established by Guo [12].
\end{remark}

\begin{remark}
Let $\mathcal{F}=\{\Omega,\emptyset\}$, then $\left(H_{\mathcal{F}}^{\Phi}(E),|||\cdot|||_{\Phi L}\right)$ and $\left(L_{\mathcal{F}}^{\Psi}\left(E^{*}\right),|||\cdot|||_{\Psi O}\right)$ are exactly $\left(H^{\phi}(E),\|\cdot\|_{\phi L}\right)$ and $\left(L^{\psi}\left(E^{*}\right),\|\cdot\|_{\psi O}\right)$, respectively, with extend one of the results in 2022.

Further, if we take $\mathcal{F}=\{\Omega,\emptyset\}$ and $\Phi(t)=t^p$ for all $p \in[1, \infty)$ in Theorem 4.1, then we obtain
$$\left(L^p(E),\|\cdot\|_p\right)^{\prime} \cong\left(L^q\left(E^*\right),\|\cdot\|_q\right).$$
Thus the results extend one of the results established by Guo [15].
\end{remark}

\begin{remark}
If we take $(E,\|\cdot\|)=(L^{0}(\mathcal{E}),|\cdot|)$, then it is obvious that $$L_{\mathcal{F}}^{\Phi}(E)=L_{\mathcal{F}}^{\Phi}(\mathcal{E})=:\left\{x \in L^0(\mathcal{E}): |x| \in L_{\mathcal{F}}^{\Phi}\right\}.$$ Thus $\left(L_{\mathcal{F}}^{\Phi}(\mathcal{E}),|||\cdot|||_{\Phi L}\right)$ is an $RN$ module.
\end{remark}

\begin{corollary}
Assume that $\Phi$ and $\Psi$ satisfy the $\Delta_2$-condition, then we obtain  $L_{\mathcal{F}}^{\Phi}(E)=H_{\mathcal{F}}^{\Phi}(E)$ and $L_{\mathcal{F}}^{\Psi}\left(E^*\right)=H_{\mathcal{F}}^{\Psi}\left(E^*\right)$. Using Theorem $4.1$, we have
$$
\left(L_{\mathcal{F}}^{\Phi}(E),\||\cdot\||_{\Phi L}\right)^{*} \cong\left(L_{\mathcal{F}}^{\Psi}\left(E^*\right),\||\cdot\||_{\Psi O}\right).
$$
\end{corollary}

\section*{Acknowledgment}
The study was supported by the National Natural Science Foundation of China (Grant No. 12171361) and the Humanities and Social Science Foundation of Ministry of Education (Grant No. 20YJC790174).

\end{document}